\documentclass[fontsize=11pt,reqno]{amsart}

\usepackage{amsmath}
\usepackage{amsthm}
\usepackage{amssymb}
\usepackage{mathtools}
\usepackage{mathrsfs}
\usepackage{thmtools}
\usepackage{thm-restate}
\usepackage{physics}
\usepackage[margin=1in]{geometry}
\usepackage{scrextend}
\usepackage[shortlabels]{enumitem}
\usepackage{colonequals}
\usepackage{comment}
\usepackage{setspace}
\usepackage{subfiles}
\usepackage{tikz-cd}
\usepackage{enumitem}

\usepackage[notcolon, notquote]{hanging}

\usepackage[title]{appendix}

\usepackage{hyperref}
\hypersetup{
    colorlinks=true,
    linkcolor=blue,
    citecolor=blue,
    }
\usepackage[noabbrev,capitalize]{cleveref}
\newtheorem{thm}{Theorem}[section]
\newtheorem{theorem}[thm]{Theorem}

\newtheorem{lemma}[thm]{Lemma}
\newtheorem{proposition}[thm]{Proposition}

\theoremstyle{definition}

\theoremstyle{remark}

\newtheorem{setup}[thm]{Set-up}

\newcommand{\C}{\mathbb{C}}
\newcommand{\tC}{\widetilde{C}}

\renewcommand{\P}{\mathbb{P}}
\renewcommand{\O}{\mathcal{O}}

\newcommand{\Sym}{\mathrm{Sym}}

\newcommand{\Jac}{\mathrm{Jac}}

\title{Low Degree Points on Singular Plane Curves}
\author[Canale, Chen, Curewitz, Daum, Dovgodko, Santiago-Calder\'{o}n, Yajnik]{Zachary R. Canale, Nathan Chen, Zoe Curewitz, Jacob A. Daum, \\Karina Dovgodko, Carlos F. Santiago-Calder\'{o}n, and Shiv R. Yajnik}

\begin{document}

\begin{abstract}
The purpose of this paper is to study low degree points on plane curves. We prove results analogous to those of Debarre and Klassen for singular plane curves with a finite number $\delta$ of ordinary nodes/cusps, where $\delta$ is bounded from above by a quadratic function in the degree of the plane curve.
\end{abstract}

\maketitle

%%%%%%%%
%%  INTRODUCTION
%%%%%%%%

\section{Introduction}

Over the past thirty years, there has been a flurry of development in the study of low-degree points on curves. In 1991, Faltings proved the Mordell--Lang conjecture, which characterizes when an algebraic curve can contain infinitely many points of bounded degree in terms of geometric properties. Shortly afterwards, there were a number of interesting applications to the existence of low-degree points on curves, beginning with the work of Abramovich \cite{Abramovich91}, Abramovich--Harris \cite{AH91}, Debarre--Fahlaoui \cite{DF93}, Debarre--Klassen \cite{DK94}, Harris--Silverman \cite{HS1991}, and many others.

Recall that the \textit{degree} of an algebraic point $p \in C(\overline{K})$ is the degree of the field extension of its residue field. In \cite{DK94}, Debarre--Klassen leveraged Faltings's theorem to study low-degree points on plane curves. When $C \subset \P^2$ is a \textit{nice} plane curve (i.e., smooth, projective, and geometrically integral) of degree $d \geq 7$ over a number field $K$, they show that $C$ contains only finitely many points of degree $\leq d-2$. Furthermore, they prove that if $d\geq 8$, then all but finitely many points on $C$ of degree $\leq d-1$ arise as the intersection of $C$ with a line in $\P^2$ defined over $K$ which passes through a $K$-point of $C$.

The goal of this paper is to extend the results of Debarre--Klassen to mildly singular plane curves. Our main result is the following.

\begin{restatable}{thm}{CCCDDSY}\label{thm:CCCDDSY} Let $C \subset \P^{2}$ be a plane curve over a number field $K$ with $\delta$ ordinary nodes and/or cusps. Suppose one of the following conditions holds:
\[ \{ d \geq 8 \quad \text{and} \quad \delta \leq d-8 \} \quad \text{or} \quad \left\{ d \geq 17 \quad \text{and} \quad \delta < \frac{d^{2}}{4} - 2d - \frac{17}{4} \right\}. \]
Then $C$ only contains finitely many points of degree $\leq d-3$, and all but finitely many points of $C$ of degree $\leq d-1$ arise as the intersection of $C$ with lines in $\P^2$ defined over $K$.
\end{restatable}

\noindent The first part of the theorem can be rephrased as saying that the minimum $\min \delta(C/K)$ of the density degree set is $\geq d-2$ (see \S2). In addition, $d-2 \in \delta(C/K)$ if and only if one of the nodes/cusps is a $K$-point, and $d-1 \in \delta(C/K)$ if and only if $C$ contains a smooth $K$-point.

We refer the reader to the recent work of Kadets--Vogt, where they study low degree points on arbitrary curves under a certain genus assumption \cite{KV25}. This genus assumption is not satisfied for the curves that we consider in Theorem~\ref{thm:CCCDDSY} (one can use the degree-genus formula for a plane curve to see this). This connection is further clarified in \S4, where we present an example due to Coppens--Kato that demonstrates that the quadratic bound on $\delta$ in Theorem~\ref{thm:CCCDDSY} is asymptotically optimal. After a draft of this paper was written, we were informed of some related work of Granot \cite{Granot2025}, where they study low degree points on curves in toric surfaces.

\subsection*{Acknowledgments}
This project originally began as an REU at Columbia University over the summer of 2023. We would like to thank Avi Zeff for his guidance and mentorship. We would also like to thank George Dragomir for organizing the REU program and making this research opportunity possible. During the preparation of this article, NC was partially supported by an NSF postdoctoral fellowship DMS-2103099.

%%%%%%%%
%%  BACKGROUND AND REDUCTION STEP
%%%%%%%%

\section{Background}\label{Section:Background}

Recall that the \textit{density degree set} $\delta(C/K)$ is the set of positive integers $e$ for which the set of degree $e$ points on $C$ is Zariski dense (cf. \cite[\S 2]{VV24}) -- since $C$ is one-dimensional, note that this is equivalent to the existence of infinitely many degree $e$ points. The degree $e$ \textit{Jacobian variety} $\Jac^{e}(C)$ denotes the connected component of the Picard group $\mathrm{Pic}(C)$ which consists of degree $e$ line bundles. For $e \geq 1$, $\Sym^{e}(C)$ will denote the $e$-th symmetric product of $C$. The degree $e$ \textit{Abel--Jacobi map}
\[ \mathrm{AJ}_{e}: \Sym^e(C) \to \Jac^e(C) \]
sends an effective degree $e$ zero-cycle $D$ to its divisor class $[D]$. We denote by $W_e(C) \subset \Jac^e(C)$ the image of $\Sym^e(C)$ under $\mathrm{AJ}_{e}$.

Given a complete linear system $|D|$ of degree $m$ on $C$, we denote by $g^r_m$ a projective linear subspace of $|D|$ of dimension $r$.  A $g^1_m$ is called \textit{a pencil}. For a singular plane curve $C$ with ordinary nodes and/or cusps and normalization $\tC$, we will say that a linear series on $\tC$ is \textit{induced by} curves on $C$ if it is the restriction of a linear system of plane curves (of some fixed degree) -- see the second paragraph of \cite{CK91}.

\begin{setup}\label{Setup}
    Fix a number field $K$ and let $C \subset \P^{2}$ be an integral plane curve over $K$ of degree $d$ with $\delta$ ordinary nodes and/or cusps (and no other singularities). Let $\nu \colon \tC \rightarrow C$ denote the normalization.
\end{setup}

The starting point for the work of Debarre--Klassen \cite{DK94} is the following remarkable result of Faltings, which proved the Mordell-Lang conjecture:

\begin{theorem}[\cite{Faltings91}]\label{thm:Faltings} Let $A$ be an abelian variety defined over a number field $K$. If $X$ is a subvariety of $A$ that does not contain any translate of a positive-dimensional abelian subvariety of $A$, then $X$ contains only finitely many $K$-points.\end{theorem}

\noindent Following ideas of Abramovich-Harris \cite{AH91} and Debarre--Klassen \cite{DK94}, in order to understand low degree points on a plane curve $C \subset \P^{2}$, one can reduce to studying rational points on certain $W_{e}(\tC)$ varieties which naturally live inside the Jacobian of the normalization $\tC \rightarrow C$. More precisely, degree $e$ points can be viewed as rational points on the $e$-th symmetric power of $C$. Using the geometry of the Abel--Jacobi map, our main result reduces to a statement about the presence of translates of positive-dimensional abelian subvarieties in $W_e(\tC)$, the image of the symmetric product $\Sym^e(\tC)$ under the Abel--Jacobi map. The main difficulty is that we must work with the normalization $\tC$ of $C$; although $\tC$ is smooth, its linear series are different from that of a smooth plane curve.

Suppose that $\tC$ has infinitely many distinct degree $e$ points. Then at least one of two possibilities must occur: (a) infinitely many distinct $K$-points lie inside of a fiber of the Abel--Jacobi map, which is a linear series, or (b) the image $W_{e}(\tC)$ contains infinitely many distinct $K$-points. Our strategy, following the approach in \cite{DK94}, is to show that in a certain range $1 \leq e \leq d-1$, the variety $W_{e}(\tC)$ does not contain any positive dimensional abelian subvarieties. Then (a) must hold; in other words, the $K$-points come from some linear series.

In order to describe these linear series on $\tC$ more precisely, we will first derive a number of corollaries of the main results of Coppens--Kato \cite{CK91, CK91correction}, which we restate here.

\begin{theorem}[\cite{CK91}, Theorems 2.1, 2.3, and 2.4(ii)]\label{thm:CKcompilation}
    Fix a positive integer $k$. Let $C$, $\tC$ be as in Set-up~\ref{Setup}.
    \begin{enumerate}[topsep=-1ex, itemsep=-0.5ex, partopsep=1ex, parsep=1ex]
        \item[(i)] If $d \geq 2(k+1)$ and $\delta < kd - (k+1)^2 + 3$, then $\tC$ has no linear system $g_{e}^1$ with $e \leq d-3$.
        \item[(ii)] If $d \geq 2k+3$ and $\delta < kd - (k+1)^{2}+2$, then each linear system $g_{d-2}^{1}$ on $\tC$ is cut out by a pencil of lines.
        \item[(iii)] If $k \geq 3$, $d \geq 2k+3$, and $\delta < kd - (k+1)^{2} + 1$, then each $g^{1}_{d-1}$ on $C$ is cut out by a pencil of lines.
    \end{enumerate}
\end{theorem}

We can produce explicit bounds on the number of nodes by solving for the optimal value of $k$.

\begin{proposition}\label{prop:corollaryOfCK}
    Let $C$ and $\tC$ be as in Set-up~\ref{Setup}.
    \begin{enumerate}[topsep=-1ex, itemsep=-0.5ex, partopsep=1ex, parsep=1ex]
        \item[(i)]\label{prop:corollaryOfCK-1} If $d \geq 4$ and $\delta < \frac{1}{4} d^{2} - d + \frac{11}{4}$, then $\tC$ does not admit a $g^{1}_{e}$ with $e \leq d-3$.
        \item[(ii)]\label{prop:corollaryOfCK-2} If $d \geq 5$ and $\delta \leq \frac{d^{2}}{4} - d$, then every $g^{1}_{d-2}$ on $\tC$ is induced by a pencil of lines.
        \item[(iii)]\label{prop:corollaryOfCK-3} If $d \geq 9$ and $\delta \leq \frac{d^{2}}{4} - d - 1$, then every $g^{1}_{d-1}$ on $\tC$ is induced by a pencil of lines.
    \end{enumerate}
\end{proposition}

\begin{proof}
    The proofs for each part are quite similar, and will rely on results of Coppens--Kato.

    (i) Consider the function $f(\ell) \coloneqq \ell d - (\ell+1)^{2}+3$ and note that it is increasing on the interval $1 \leq \ell \leq \frac{d}{2}-1$. We will set $k \coloneqq \lfloor \frac{d}{2} \rfloor - 1$ and apply Theorem~\ref{thm:CKcompilation}(i). In order to do so, we need to check that $\delta < f(k)$. One can verify that this is true since
    \[ \delta < \frac{1}{4} d^{2} - d + \frac{11}{4} = f\left( \frac{d-3}{2} \right) \leq f(k). \]
    For the last inequality, it is helpful to note that $\frac{d-3}{2} \leq k$.
    
    (ii) The assumption $d \geq 5$ ensures that the right hand side of the bound for $\delta$ is non-negative. Now let $f(\ell) \coloneqq \ell d - (\ell+1)^{2}+2$ and fix $k = \lfloor \frac{d-1}{2} \rfloor - 1$. In order to apply ~\cref{thm:CKcompilation}(ii), we need to check that $\delta < f(k)$. By assumption, we have
    \[ \delta \leq \frac{d^{2}}{4} - d < f\left( \frac{d}{2} - 2 \right) \leq f(k), \]
    where the last inequality follows from the fact that $f(\ell)$ is increasing on the interval $1 \leq \ell \leq \frac{d}{2}-1$. 

    (iii) Set $f(\ell) \coloneqq \ell d - (\ell+1)^{2}+1$ and fix $k = \lfloor \frac{d-1}{2} \rfloor - 1 \geq 3$ (here is where we need $d \geq 9$). In order to apply ~\cref{thm:CKcompilation}(iii), we need to check that $\delta < f(k)$. Following the same strategy as in part (ii), one has
    \[ \delta \leq \frac{d^{2}}{4} - d - 1 < f\left( \frac{d}{2} - 2 \right) \leq f(k). \qedhere \]
\end{proof}

We will also need to recall two more lemmas of Coppens--Kato. Let $\P_{m}$ denote the projective space parametrizing plane curves of degree $m$.

\begin{lemma}[\cite{CK91correction}, Main Lemma]\label{lemma:CKMain}
Fix a positive integer $k$. Following Set-up~\ref{Setup}, if $n+\delta < k (d - k)$ for some integer $k>0$, then for any base-point-free $g^1_n$ on $\tC$, there exists a pencil $\P \subset \P_{k-1}$ inducing $g^1_n$ on $C$.
\end{lemma}

\begin{lemma}[\cite{CK91}, Lemma 1.2]\label{lemma:CK1.2}
Following Set-up~\ref{Setup}, consider a linear system $g^1_n$ on $\tC$ induced by a pencil $\P \subset \P_{m}$ without fixed components. If there are $\varepsilon$ points among the $\delta$ singularities of $C$ which are part of the base locus of $\P$, then
\[ n\geq md-m^2-\varepsilon. \]
Note that by Bezout's theorem, $\varepsilon \leq m^{2}$.
\end{lemma}

\noindent When we say that the pencil $\P$ above does not have fixed components, this means that it is not of the form $F_{k} + \P_{m'}$ where $F_{k}$ is a fixed curve and $\P_{m'}$ is another pencil of degree $m'$ curves in $\P^{2}$.

Using these lemmas, we will show the following:

\begin{proposition}\label{prop:quadBoundConics}
    Following Set-up~\ref{Setup}, suppose
    \[ d \geq 17 \quad \text{and} \quad \delta < \frac{d^{2}}{4} - 2d - \frac{17}{4}. \]
    Then any base-point-free linear series $g^{1}_{n}$ on $\tC$ with $n \leq 2d-2$ comes from either a pencil of lines or a pencil of conics. If it comes from a pencil of conics, then $2d-8 \leq n \leq 2d-2$.
\end{proposition}

\begin{proof}
    Define $g(\ell) = \ell (d-\ell)$ and set $k = \left\lfloor \frac{d}{2} - 2 \right\rfloor$. Then we have a series of inequalities
    \[ n+\delta < \frac{d^{2}}{4} - \frac{25}{4} = g \left( \frac{d}{2}- \frac{5}{2} \right) \leq g(k), \]
    where the last inequality follows from the fact that $g(\ell)$ is increasing for $0 \leq \ell \leq \frac{d}{2}$. By ~\cref{lemma:CKMain}, it follows that the $g^{1}_{n}$ is cut out by a pencil of degree $m$ plane curves for some $1 \leq m \leq k-1$. By ~\cref{lemma:CK1.2}, note that $n \geq md - 2m^{2}$.

    Suppose for contradiction that $m \geq 3$. From our assumptions and the previous paragraph, we have the three inequalities:
    \begin{enumerate}[(i)]
        \item $n \leq 2d-2$,
        \item $n \geq md - 2m^{2}$,
        \item $1 \leq m \leq k-1 \leq \frac{d}{2} - 3$.
    \end{enumerate}
    From (i) and (ii), we deduce that
    \begin{equation}\label{eq:upperboundond}
    d \leq \frac{2(m^2-1)}{m-2}.
    \end{equation}
    The assumption that $d \geq 17$ then implies that $m \geq 6$. Further simplifying \eqref{eq:upperboundond} yields
    \[ d \leq \frac{2(m^2-1)}{m-2} = 2m + 2 + \frac{2 \cdot (m+1)}{m-2}, \]
    where the last fraction is $\leq \frac{7}{2} < 4$ when $m \geq 6$. On the other hand, (iii) implies that $d \geq 2m + 6$. This gives the desired contradiction.
    
    Thus, we conclude that $m = 1$ or $2$, which means that the $g^{1}_{n}$ is induced by either a pencil of lines or a pencil of conics. In the former case, $g^{1}_{n} = g^{1}_{e} + D_{n-e}$ where $e \leq d$. In the latter case, a pencil of conics has a base locus of 4 points, so there are at most four points on $C$ which are fixed by the pencil of conics. Depending on whether these points are smooth or nodes/cusps, it follows that $g^{1}_{n} = g^{1}_{e} + D_{n-e}$ where $2d-8 \leq e \leq 2d$. The same lower bound holds for $n$.
\end{proof}

Finally, we will conclude this section by extending the results of Proposition~\ref{prop:quadBoundConics} to plane curves of small degree.

\begin{proposition}\label{prop:linearBoundConics}
    In the setting of Set-up~\ref{Setup}, let $H$ be the pullback of the hyperplane section to $\tC$. Consider an effective divisor $D$ which belongs to a base-point-free linear series $g^{1}_{n}$ (where $2\leq n \leq 2d-2$).
    
    \begin{enumerate}[topsep=-1ex,itemsep=-0.5ex,partopsep=1ex,parsep=1ex]
        \item[(i)] If $d\geq 5$, $\delta \leq d-1$, and $n\leq 2d-9$, then in fact $d \geq 7$ and the $g^{1}_{n}$ comes from a pencil of lines.
        
        \item[(ii)] If $d\geq 8$, $\delta \leq d-8$ and $2d-8 \leq n \leq 2d-2$, then the $g^{1}_{n}$ comes from a pencil of conics and we can write $D = 2H - E$ for some fixed divisor $E$ which is the pullback via the normalization $\nu$ of up to four (possibly singular) points on $C$.
    \end{enumerate}
\end{proposition}

\begin{proof}
We observe the following:

(i) Since $\delta \leq d-1$ and $n\leq 2d-9$, by applying \cite[Main Lemma]{CK91correction} with $k=3$ it follows that the $g^{1}_{n}$ comes from a pencil of degree $m$ curves with $m \leq 2$. By \cref{lemma:CK1.2}, $n \geq md - 2m^{2}$, which combined with $n \leq 2d-9$ implies that we must have $m = 1$. The claim $d \geq 7$ is due to the fact that $n \geq d-2$ by Proposition~\ref{prop:corollaryOfCK}(i).
    
(ii) The statement for $d = 8$ was already proved by Debarre--Klassen, so we may assume $d \geq 9$. Since $\delta \leq d-8$ and $2d-8 \leq n \leq 2d-2$, by applying \cite[Main Lemma]{CK91correction} with $k=3$ it follows that the $g^{1}_{n}$ comes from a pencil of degree $m$ plane curves with $m \leq 2$. Finally, $d \geq 9$ implies $n \geq 2d-8 > d$. But the $g^{1}_{n}$ is base-point-free, so this forces $m = 2$. \qedhere
\end{proof}

%%%%%%%%
%%  NONEXISTENCE OF ABELIAN SUBVARIETIES
%%%%%%%%

\section{Nonexistence of abelian subvarieties}\label{Section:MainProof}

In this section, we determine whether the image $W_e(\tC)$ of the Abel--Jacobi map associated with the normalization $\tC$ contains abelian subvarieties. It may be beneficial to compare Proposition \ref{prop:CCCDDSY.AbelianVarieties} with Proposition 1 from Section 5 of \cite{DK94}. We will then deduce Theorem~\ref{thm:CCCDDSY}.

\begin{proposition}\label{prop:CCCDDSY.AbelianVarieties}
    In the setting of Set-up~\ref{Setup}, assume that one of the following sets of conditions holds:
    \[ \{ d \geq 8 \quad \text{and} \quad \delta \leq d-8 \} \quad \text{or} \quad \{ d \geq 17 \quad \text{and} \quad \delta \leq \frac{d^{2}}{4} - 2d - 3 \}. \]
    Then $W_{e}(\tC)$ contains no positive-dimensional abelian subvarieties for $e\leq d-1$.
\end{proposition}

\begin{proof}
    The case where $d = 8$ and $\delta = 0$ has been proved in \cite{DK94}, so from now on let us assume $d \geq 9$. Let $1 \leq e \leq d-1$ and suppose $W_{e}(\tC)$ contains an abelian subvariety $A$ of dimension $h > 0$ (up to translation). Consider the diagram
    \begin{center}
    \begin{tikzcd}
    {\Sym^{e}(\tC) \times \Sym^{e}(\tC)} \arrow[d, swap, "\mathrm{AJ}_{e} \times \mathrm{AJ}_{e}"] \arrow[r, "\sum"] & \Sym^{2e}(\tC) \arrow[d, "\mathrm{AJ}_{2e}"] &[-0.2in] \\
    {W_{e}(\tC) \times W_{e}(\tC)} \arrow[r, "\otimes"] & W_{2e}(\tC) \\
    {A \times A} \arrow[u, hook] \arrow[r] & A_{2} \arrow[u, hook]
    \end{tikzcd}
    \end{center}
If we let $A_{2} \subset W_{2e}(\tC) \subset \Jac^{2e}(\tC)$ denote the image of $A \times A$, then $A_{2}$ is in fact a translate of $A$ of dimension equal to $\dim A = h$, since the map from $A \times A$ to $A_{2}$ is induced by the binary operation on $\Jac(\tC)$. Let $[M] \in A_2$ be a general point, and let $r$ denote the dimension of the linear series $|M|$. By Lemma 8 of \cite{Abramovich91}, $r\geq h$. By induction on $e$, we may assume $A$ is not contained in a copy of $x+W_{e-1}(\tC)$. Therefore, the linear series $|M|$ is base-point-free.

The proof consists of three steps. (a) First, we show that $W_{e}(\tC)$ contains no abelian subvarieties of dimension $h \geq 2$ for $e\leq d-1$. The case where $h = 1$ then splits into two further subcases: (b) $e \leq d-5$ and (c) $d-4 \leq e \leq d-1$.

(a) Assume $h\geq 2$. We get a family of base-point-free linear systems $|M|$ of dimension $r\geq 2$ and degree $2e\leq 2d-2$ parametrized by the abelian variety $A_2$ (which has dimension $\geq 2$). The divisors in the linear system $|H|$ come from the same line bundle (up to isomorphism), and the linear system induced by a family of lines through a (possibly singular) point $x \in C$ has dimension $1$. Thus, from our degree assumptions and Propositions~\ref{prop:quadBoundConics} and \ref{prop:linearBoundConics}, it follows that the linear series $|M|$ is induced by a family of conics and we can write $|M| = |2H - \nu^{\ast} E|$ for some nonempty effective divisor $E$ on $C$ which represents the base locus of the conics.

Since the base locus of the family of conics may consist of either singular or smooth points on $C$, to clarify the situation, let $x_{i} \in C$ denote the points in the base locus that are singular points of $C$ and let $y_{j} \in C$ be the smooth points ($y_{j}$ may be identified with points $y_{j} \in \tC$). Then we may write
    \[ |M| = \left| 2H - \sum_{i \in I} \nu^{\ast}(x_{i}) - \sum_{j \in J} y_{j} \right|, \]
for some finite indexing sets $I$ and $J$, where $\# I + \#J$ is the total number of base points of the family of conics in $\P^{2}$. Here, $\nu^{\ast}(x_{i})$ is a degree two divisor on $\tC$ which is either the two points in the preimage of a node, or a single point with multiplicity two which is the preimage of a cusp. We know that $\# I + \#J \leq 3$ because otherwise $|M|$ comes from a pencil of conics with four base points (any three must necessarily be non-collinear) and therefore has dimension $\leq 1$.

Since $\deg M = 2e$ is even and $2 \leq 2e \leq 2d-2$, the remaining cases to consider are $(\#I, \#J) \in \{ (3, 0), (2, 0), (1, 0), (1, 2), (0, 2) \}$. The singular points $x_{i}$ cannot vary, so the assumption that the family has dimension $h \geq 2$ as $|M|$ varies rules out $(3, 0), (2, 0)$, and $(1, 0)$. The cases where $(\#I, \#J) = (1,2)$ and $(0,2)$ can be treated the same way, since in both cases one can define a map
    \[ A_{2} \dashrightarrow \Sym^{2}(\tC) \quad \text{where} \quad |M| \mapsto \sum_{j \in J} y_{j}. \]
This map is generically injective since $A_{2}$ embeds into $W_{e}(\tC)$, and thus (by dimension reasons) it is dominant. But this contradicts the fact that $\Sym^{2}(\tC)$ is of general type since $g(\tC) \geq 3$. One way to see this last fact is to compute the self-intersection of the canonical bundle of $\Sym^{2}(\tC)$ as well as its second plurigenus, and then apply the Kodaira-Enriques classification.

(b) Now assume $h=1$ and $n=2e\leq 2d-9$. In this case, by our degree assumption and Propositions \ref{prop:quadBoundConics} and \ref{prop:linearBoundConics}, $|M|$ is a $g^1_n$ that comes from a pencil of lines. If the linear series $|M|$ has larger dimension, one can always choose a sub-linear series. For the proof below, we may assume without loss of generality that $|M|$ is a pencil. Following the notation in part (a),
\[ |M|=\Big|H-\sum_{i\in I} \nu^*(x_i)-\sum_{j \in J} y_j\Big| \]
where the $x_i$ are the singular points in the base locus of the pencil of lines and the $y_j$ the smooth points. Also, note that the base locus consists of at most one point, since otherwise $|M|$ comes from a fixed line and hence has dimension $0$. We see that either $|M|=|H|$, $|M|=|H-\nu^*(x)|$ for some singular $x \in C$ or $|M|=|H-y|$ for some smooth $y \in C$. Note that $|M| \not= |H|$ because the linear equivalence class of the line bundle $H$ is not changing, which would imply $h = 0$, a contradiction. In the second case, there are only finitely many linear equivalence classes of line bundles that are of the form $H-\nu^*(x)$ for some singular point $x \in C$, contradicting $h=1$. In the third case, $|M|$ moves in a family parametrized by $C_{\mathrm{reg}}$, implying $C$ is birational to an elliptic curve, another contradiction.

(c) Finally, assume $h=1$ and $d-4 \leq e \leq d-1$. Then again by our degree assumptions and Propositions \ref{prop:quadBoundConics} and \ref{prop:linearBoundConics}, $|M|$ takes the form
\[ |M| = \Big| 2H - \sum_{i \in I} \nu^{\ast}(x_{i}) - \sum_{j \in J} y_{j} \Big|, \]
where $\deg (\sum_{i \in I} \nu^{\ast}(x_{i}) + \sum_{j \in J} y_{j}) = 2\cdot \#I + \#J \in \{ 2, 4, 6, 8 \}$. Here we are again follow the notation in part (a). The divisor $\sum_{i \in I} \nu^{\ast}(x_{i}) + \sum_{j \in J} y_{j}$ corresponds to the base locus of a pencil of conics, and consists of points where no three are collinear. By a similar argument as in the previous step, since the linear series must vary in a positive-dimensional family, $\#J\geq 1$. On the other hand, two conics intersect at four points counted with multiplicity, so $\#I + \#J\leq 4$. So the remaining cases are
    \[ (\#I, \#J) \in \{(0, 2),\;(1, 2),\;(2, 2),\;(0, 4)\}. \]
In the cases where $\#J=2$, we can define a generically injective map $A_2 \dashrightarrow \text{Sym}^2(\tC)$. By \cite[Theorem 2(a)]{HS1991}, it follows that $\tC$ is hyperelliptic or bielliptic. A hyperelliptic curve has a $g^1_2$, and $\tC$ has no $g^1_n$ for $n\leq d-3$ so that $\tC$ must be bielleptic. But the composition $\tC \stackrel{2:1}{\longrightarrow} X \stackrel{2:1}{\longrightarrow} \P^1$ then gives a $g^1_4$ on $\tC$, so using Proposition~\ref{prop:corollaryOfCK}(i) we can rule out every case except for $(0,4)$.
    
For $(\#I, \#J) = (0,4)$, we can define a nonconstant map $A_2 \to W_k(\tC)$ ($1 \leq k\leq 4$) by sending
    \[ |M|=\Big|2H-\sum_{j=1}^4 y_j\Big| \longmapsto AJ \Big( \sum_{j \in J'} y_j \Big)\]
where $J'$ is the indexing set of the $y_j$s that vary as $|M|$ varies and $k \coloneqq \# J'$. Lemma 8 of \cite{Abramovich91} then implies that the image of the map $(A_2)_2 \rightarrow W_{2k}(C)$ is an elliptic curve in $W_{2k}(C)$ that parametrizes pencils of degree $2k \leq 8 \leq 2d-9$ (where the second inequality follows from the fact that $d \geq 9$). This directly contradicts (b). \qedhere

\end{proof}
Now we are ready to prove \cref{thm:CCCDDSY}. Recall the statement of the main theorem:

\CCCDDSY*
\begin{proof}
    It suffices to prove the theorem for the normalization $\tC$, because $C$ contains only finitely many singularities and the normalization is an isomorphism away from these singularities.

    As mentioned in the introduction, a degree $e$ point on $\tC$ (together with its conjugates) gives rise to a $K$-point on $\Sym^e(\tC)$. First, we claim that for every $1 \leq e \leq d-3$, $\Sym^{e}(\tC)$ contains only finitely many $K$-points. By our degree assumption on $e$ and Proposition~\ref{prop:corollaryOfCK}(i), $\Sym^e(\tC)$ is isomorphic to its image $W_e(\tC)$ under the Abel--Jacobi map. Now \cref{prop:CCCDDSY.AbelianVarieties} implies that $W_{e}(\tC)$ does not contain any abelian subvarieties, so by Faltings's theorem (\cref{thm:Faltings}) it follows that $W_{e}(\tC)$ can only contain finitely many $K$-points. 

    When $d-2 \leq e \leq d-1$, the same argument as above shows that $W_e(\tC)$ still has finitely many $K$-points. Since the image of a $K$-point is a $K$-point, all but finitely many $K$-points in $\Sym^e(\tC)$ must be contained in finitely many fibers of the Abel-Jacobi map
    \[ \Sym^{e}(\tC) \rightarrow W_{e}(\tC) \subset \Jac^{e}(\tC). \]
    These fibers are positive-dimensional linear series on $\tC$ (note that if all of these linear series are zero-dimensional then we are done). Let $D \in |M|$ be a $K$-point of degree exactly $e$ in one of these positive-dimensional linear series $|M|$, and choose a pencil $g^{1}_{e}$ on $\tC$ passing through $D$ (it may have base-points a priori). We will show that this pencil comes from a family of lines passing through a point $x \in C$.
    
    Let $e = d-2$. Then the pencil $g^{1}_{e}$ is actually base-point-free by Proposition~\ref{prop:corollaryOfCK}(i), so Proposition~\ref{prop:corollaryOfCK}(ii) implies that the $g^{1}_{e}$ is induced by a pencil of lines (both sets of hypotheses in the theorem statement are covered). Since $e = d-2$, the lines must pass through a singular $K$-point on $C$ of multiplicity 2, i.e. one of the nodes or cusps. Thus $D$ comes from the intersection of $C$ with a line in $\P^{2}$.

    Now we will deal with the case $e = d-1$. By Proposition~\ref{prop:corollaryOfCK}(i), the pencil $g^{1}_{e}$ has at most one fixed-point, which would necessarily have to be a $K$-point $P$ on $\tC$. If it actually has a fixed point, then $D$ is actually a point of degree $d-2$ and not degree $e = d-1$. Thus, without loss of generality we may assume the pencil $g^{1}_{e}$ is base-point-free. For $d \geq 9$, by  Proposition~\ref{prop:corollaryOfCK}(iii), and for $d = 8$, by \cite[Theorem 2]{DK94},
    the $g^{1}_{e}$ is induced by a pencil of lines. Since $e = d-1$, the lines must pass through a smooth $K$-point on $C$.
\end{proof}

\section{An asymptotically sharp example}

In the previous sections, we proved that results analogous to those of Debarre--Klassen \cite{DK94} hold for plane curves of degree $d$ with $\delta$ ordinary singularities if either $d\geq 8$ and $\delta \leq d-8$, or $d \geq 17$ and $\delta \leq \frac{d^2}{4} - 2d - \frac{17}{4}$. In this section, we will show that the quadratic bound for $\delta$ is close to optimal by writing down an explicit plane curve $C$ of even degree $d = 2k+4$ with $(k+1)^{2}$ nodes and no other singularities. This example is essentially due to Coppens--Kato \cite[Example 5.1]{CK91}:

\textbf{Ex.} Consider the polynomials
\begin{align*}
P(x, y, a, b) &= a \prod_{m=0}^{k} (x-m) + b \prod_{m=0}^{k} (y-m), \\
P_{1}(x, y, \alpha, \beta) &= \alpha \prod_{m=0}^{k+1} (x-m) + \beta \prod_{m=0}^{k+1} (y-m).
\end{align*}
For suitable choices of constants $a_{i}$, $b_{i}$, $\alpha_{i}$, $\beta_{i}$, and suitable choices of lines $L_{i}$, the plane curve $C$ given by the affine equation
\[ P(x,y,a_1,b_1) P_{1}(x,y,\alpha_1,\beta_1) L_{1} + P(x,y,a_2,b_2) P_{1}(x,y,\alpha_2,\beta_2) L_{2} = 0 \]
is irreducible of degree $2k+4$. Note that $C$ has ordinary nodes at the points $(i, j)$ for $0 \leq i, j \leq k$ and no other singularities. Irreducibility then follows from the fact that if $C$ is cut out by $Q_{1} \cdot Q_{2} = 0$, then it would have the wrong number of singular points.

Moreover, $C$ contains the points $(i, j)$ for $0 \leq i, j \leq k+1$, and the pencil cut out by the linear system of curves $\{ P_{1}(x, y, \alpha, \beta) = 0 \}_{[\alpha:\beta] \in \P^1}$ is a $g^{1}_{d-1}$. As a consequence of Hilbert's Irreducibility Theorem \cite[Corollary 2.8]{VV24}, the $g^{1}_{d-1}$ gives rise to infinitely many points of degree $d-1$ on the normalization $\tC$ which do not arise from lines, so the second part of Theorem~\ref{thm:CCCDDSY} fails for $C$.

Since there are $(k+1)^{2}$ nodes in total, the total difference between the arithmetic and geometric genus of the curve $C$ is
\[ \delta = (k+1)^{2} = \left( \frac{d}{2} - 1 \right)^{2} = \frac{d^{2}}{4} - d + 1, \]
which shows that the quadratic term in the second $\delta$ bound in Theorem~\ref{thm:CCCDDSY} is correct, but the linear term has a discrepancy of approximately $d + 5$.

\subsection{Closing remarks}

We would like to point out that a version of \cref{thm:CCCDDSY} with slightly different asymptotics follows from the work of Smith--Vogt \cite{SmithVogt22}. 

A key step in their argument is to show that for a nice surface $S/\mathbb{C}$ with $H^1(S,\mathcal{O}_S)=0$, if $C$ is a smooth ample curve on $S$ and $e < \frac{C^2}{9}$, then the locus $W_e(C)$ contains no positive-dimensional abelian varieties \cite[Theorem 2.9]{SmithVogt22}. In fact, with some minor modifications the ampleness condition on $C$ can be relaxed to allow $\O_{S}(C)$ to be just big and nef. To prove this, we need the following analogue of \cite[Proposition 2.5]{SmithVogt22}:

\begin{proposition}\label{prop:ExistsDivisor}
Let $S$ be a smooth projective surface and $C\subset S$ a smooth, integral curve such that $\mathcal{O}_S(C)$ is nef. If $\Gamma$ is a divisor on $C$ that moves in a basepoint-free pencil, satisfying \begin{equation*}
    \text{deg } \Gamma < C^2/4,
\end{equation*} 
then there exists a divisor $D$ on $S$ satisfying the following four properties:
\begin{enumerate}[itemsep=0pt]
    \item $h^0(S,D)\geq 2$,
    \item $C\cdot D < C^2/2$,
    \item $\text{deg } \Gamma \geq D \cdot (C-D)$,
    \item If $E$ is any divisor on $S$ such that
\begin{equation*}\label{eConditions}
h^0(\O_C(E|_C - \Gamma)) = 0 \qquad \text{and} \qquad E \cdot C < C^2, \end{equation*}
then $h^0(\O_S(E - D)) = 0$.  In particular, $h^0(\O_C(D|_C - \Gamma)) > 0$.
\end{enumerate}
\end{proposition}

\noindent The proof is almost identical to the one in \cite[Proposition 2.5]{SmithVogt22}, so we have chosen to omit it.

Note that in the setting of $\mu-$stability with respect to a big and nef curve class $C$ on a smooth projective surface $S$, \cite[Corollary 2.25]{GKP16} shows that any coherent sheaf $\mathcal{E}$ on $S$ still has a unique maximally destabilising subsheaf, which is semistable and saturated in $\mathcal{E}$. \cref{prop:ExistsDivisor} can be used to prove the following analogue of \cite[Theorem 2.9]{SmithVogt22}.

\begin{thm}\label{thm: main_geometric}
Let $S/\C$ be a nice surface with $h^1(S, \O_S)=0$, and let $C$ be a smooth curve on $S$ whose class is big and nef. Then for $e < C^2/9$, the locus $W_e(C)$ contains no positive-dimensional abelian varieties.
\end{thm}

This offers another approach to \cref{thm:CCCDDSY}:

\begin{restatable}{thm} {CCCDDSY2}\label{thm:CCCDDSY2} Let $C \subset \P^{2}$ be a plane curve over a number field $K$ with $\delta$ ordinary nodes/cusps. Suppose $d \geq 9$ and $0<\delta \leq \frac{d^2}{4}-\frac{9}{4}d+2$.
Then $C$ only contains finitely many points of degree $\leq d-3$, and all but finitely many points of $C$ of degree $\leq d-1$ arise as the intersection of $C$ with lines in $\P^2$ defined over $K$.
\end{restatable}
\begin{proof}
    Let $\tC$ be the strict transform of $C$ inside the blowup $S \rightarrow \mathbb{P}^2$ of $\P^{2}$ at the singular points of $C$. Then $S$ is a nice surface with $H^1(S,\mathcal{O}_S)=0$, and $\tC$ is smooth since $C$ only has ordinary nodes/cusps. Furthermore, one can check that $\tC^{2} > 0$ and $\tC$ must intersect non-negatively against any other curves on $S$. By the Nakai-Moishezon criterion it follows that $\tC$ is nef, and $\tC^{2} > 0$ implies that $\tC$ is also big. If $d-1 < \frac{\tC^2}{9}$, then Theorem~\ref{thm: main_geometric} implies that $W_e(\tC)$ contains no positive-dimensional abelian varieties for $e \leq d-1$. The same argument as in the proof of \cref{thm:CCCDDSY} %(using \ref{prop:corollaryOfCK})%
    then gives us the desired result.
    
    All that remains to check is that the inequality below holds:
    \begin{equation*}
        \frac{\tC^2}{9} = \frac{d^2-4\delta}{9} > d-1 \iff d^2-9d+9-4\delta >0.
    \end{equation*}
    Note that the assumption $d \geq 9$ is used to show that $\frac{d^2}{4}-\frac{9}{4}d+2 \geq 1$.
\end{proof}

\bibliographystyle{plain}
\bibliography{bibliography}

\footnotesize{

\noindent \textsc{Department of Mathematics, CUNY Graduate Center, New York, NY 10016} \\
\noindent Zachary R. Canale, \textit{E-mail address:} \href{mailto:zcanale@gradcenter.cuny.edu}{zcanale@gradcenter.cuny.edu}

\noindent \textsc{Department of Mathematics, Harvard University, Cambridge, MA 02138} \\
\noindent Nathan Chen, \textit{E-mail address:} \href{mailto:nathanchen@math.harvard.edu}{nathanchen@math.harvard.edu}

\noindent \textsc{Department of Mathematics, Barnard College, New York, NY 10027} \\
\noindent Zoe Curewitz, \textit{E-mail address:} \href{mailto:zc2567@barnard.edu}{zc2567@barnard.edu}

\noindent \textsc{Department of Mathematics, Columbia University, New York, NY 10027} \\
\noindent Jacob A. Daum, \textit{E-mail address:} \href{mailto:jad2309@columbia.edu}{jad2309@columbia.edu}

\noindent \textsc{Department of Mathematics, Columbia University, New York, NY 10027} \\
\noindent Karina Dovgodko, \textit{E-mail address:} \noindent \href{mailto:kmd2235@columbia.edu}{kmd2235@columbia.edu}

\noindent \textsc{Department of Mathematics, Columbia University, New York, NY 10027} \\
\noindent Carlos F. Santiago-Calder\'{o}n, \textit{E-mail address:} \href{mailto:cfs2150@columbia.edu}{cfs2150@columbia.edu}

\noindent \textsc{Department of Mathematics, Columbia University, New York, NY 10027} \\
\noindent Shiv R. Yajnik, \textit{E-mail address:} \href{mailto:sry2111@columbia.edu}{sry2111@columbia.edu}

}

\end{document}